\documentclass[10pt]{amsart}

\usepackage{amsmath, amscd, amssymb,xypic}
\usepackage[frame,cmtip,arrow,matrix,line,graph,curve]{xy}
\usepackage{graphpap, color}
\usepackage[mathscr]{eucal}

\numberwithin{equation}{section}

\newtheorem{theorem}{Theorem}[section]
\newtheorem{corollary}[theorem]{Corollary}
\newtheorem{lemma}[theorem]{Lemma}

\newtheorem{proposition}[theorem]{Proposition}
\newtheorem{definition}[theorem]{Definition}
\newtheorem{remit}[theorem]{Remark}
\newtheorem{example}[theorem]{Examples}
\newtheorem{conjecture}[theorem]{Conjecture}

\numberwithin{equation}{section}

\newcommand{\pp}{\mathbb{P}}

\newcommand{\cc}{\mathbb{C}}
\newcommand{\rr}{\mathbb{R}}
\newcommand{\zz}{\mathbb{Z}}
\newcommand{\hh}{\mathbb{H}}
\newcommand{\cO}{\mathcal{O} }

\newcommand{\hk}{hyperk\"ahler }
\newcommand{\ii}{\mathbf{i} }
\newcommand{\jj}{\mathbf{j} }
\newcommand{\kk}{\mathbf{k} }

\newcommand{\grad}{{\rm grad}}
\newcommand{\cM}{\mathcal{M}}
\newcommand{\liet}{\mathfrak{t}}
\newcommand{\liets}{\mathfrak{t}^*}

\begin{document}

\title{On the cohomology of hyperk\"ahler quotients}
\date{}

\author{Lisa Jeffrey}
\address{Department of Mathematics, University of Toronto, Canada}
\email{jeffrey@math.toronto.edu}

\author{Young-Hoon Kiem}
\address{Department of Mathematics and Research Institute
of Mathematics, Seoul National University, Seoul 151-747, Korea}
\email{kiem@math.snu.ac.kr}

\author{Frances Kirwan}
\address{Mathematical Institute, Oxford University, UK}
\email{kirwan@maths.ox.ac.uk}

\thanks{LCJ was partially supported by a grant
from NSERC, YHK was partially supported by KOSEF grant
R01-2007-000-20064-0 and FCK was supported by an EPSRC fellowship.}

\begin{abstract}
This paper gives a  partial desingularization construction for hyperk\"ahler
quotients  and a criterion for the surjectivity of an analogue of
the Kirwan map to the  cohomology of
hyperk\"ahler quotients. This criterion is applied to some linear actions
on hyperk\"ahler vector spaces.

\end{abstract}

\maketitle

\section{Introduction}
Hyperk\"ahler manifolds are manifolds $M$ equipped with a Riemannian
metric $g$ and three independent complex structures $\ii,\jj,\kk$
compatible with the metric  which satisfy $\ii\jj=\kk=-\jj\ii$. They
correspondingly have three symplectic forms
$\omega_1,\omega_2,\omega_3$, or one real symplectic form $\omega_1$
and one complex symplectic form $\omega_\cc=\omega_2+i\omega_3$.
Suppose a compact connected Lie group $K$ acts on $M$ preserving the
metric and the symplectic forms. We say the action is
\emph{Hamiltonian} if there are moment maps $\mu_i$ for each
$\omega_i$. It has been an outstanding problem how much of the
package of properties of Hamiltonian group actions on symplectic
manifolds extends to
hyperk\"ahler quotients $$M/\!/\!/K:=\mu^{-1}(0)/K$$ where $\mu=(\mu_1,\mu_2,\mu_3)$. 

The first result of this paper is that the partial
desingularization construction of \cite{Kir2} extends to
hyperk\"ahler quotients. For this, we fix a complex structure, say
$\ii$, which gives us a holomorphic moment map $\mu_\cc = \mu_2 +
\sqrt{-1} \mu_3:M \to \mathfrak{k}^*\otimes_\rr \cc$, and note
that with respect to this complex structure $M/\!/\!/K$ can be
identified with the K\"{a}hler quotient of the complex subvariety
$W=\mu_\cc^{-1}(0)=\mu_2^{-1}(0)\cap \mu_3^{-1}(0)$ of $M$.

In \cite{Kir2}, a canonical way of partially resolving the
singularities of the K\"ahler quotient $M/\!/K=\mu_1^{-1}(0)/K$ is
described. The most singular points of $M/\!/K$ correspond to the
points in $\mu_1^{-1}(0)$ of largest stabilizer groups in $K$.
We
blow up $M^{ss}$ along the locus of points where
the  stabilizer has maximal
dimension; this locus turns out to be smooth.
 Here $M^{ss}$ denotes the
open subset of \emph{semistable} points in $M$,
in other words the set of points  whose $K^\cc$-orbit
closure meets $\mu_1^{-1}(0)$.
 Let $M_1$ be the
result 
of the
blow-up.
Then one finds that the maximal dimension of stabilizers
on  $M_1^{ss}$ is strictly smaller than the corresponding maximal
dimension  on $M^{ss}$.
 We repeat
this process until we cannot find a semistable point whose
stabilizer is infinite. The
final result
 of this process is an analytic
variety $\widetilde{M^{ss}/\!/K}$ which has at worst orbifold
singularities, which we call the \emph{partial desingularization}
of the K\"ahler quotient.

In order to find a partial desingularization of a \hk quotient, we
take the proper transform $\widetilde{W/\!/K}$ of $W/\!/K$ in
$\widetilde{M/\!/K}$. By construction, we may describe
$\widetilde{W/\!/K}$ by the following blow-up process.
Let $H$ be  the stabilizer of a point in $\mu^{-1}(0)=W\cap
\mu_1^{-1}(0)$ for which the dimension
is maximal possible
among such, and let $H_0$ be the identity component of $H$.
Let $Z_{H_0}$ be the fixed point set of $H_0$ in $W$. We blow up
$W^{ss}$ along $GZ_{H_0}^{ss}$ to get a variety $W_1$. We repeat
this process until, after finitely many steps, there are no
semistable points with infinite stabilizers. Then we take the
K\"ahler quotient of the resulting variety $\widetilde{W}$. This is
our partial desingularization $\widetilde{W/\!/K}$.

Since $W=\mu_\cc^{-1}(0)$ can be very singular, it is by no means
obvious why the proper transform $\widetilde{W/\!/K}$ should have
at worst orbifold singularities. However, we prove the following.
\begin{theorem} 
Suppose there is at least one point in $\mu^{-1}(0)$ whose
stabilizer is finite. The above blow-up procedure produces a
surjective analytic map
$$\Pi:\widetilde{W/\!/K}\to W/\!/K$$ such that
\begin{enumerate} \item $\Pi$ is an isomorphism on the open subset of
the orbits of $x\in \mu^{-1}(0)$ where $d\mu_1$ is
surjective;
\item $\widetilde{W/\!/K}$ has at worst orbifold
singularities.\end{enumerate}
\end{theorem}
%
The advantage of this partial desingularization construction is that
it works even when the hyperk\"{a}hler moment map $\mu$ has no
central regular value close to zero (in contrast to the
method given by perturbing the hyperk\"{a}hler moment map
by adding a central constant). The disadvantage, on the other
hand, is that $\widetilde{W/\!/K}$ does not inherit a
hyperk\"{a}hler structure everywhere, and depends (only) on the
choice of complex structure $\ii$.

\bigskip

Next, we consider the restriction map, often called the Kirwan map,
\[
\kappa:H^*_K(M)\longrightarrow H^*_K(\mu^{-1}(0))\cong
H^*(\mu^{-1}(0)/K)
\]
from the equivariant cohomology of $M$ to the ordinary cohomology
of the hyperk\"ahler quotient $M/\!/\!/K=\mu^{-1}(0)/K$, when $0$
is a regular value of $\mu$.
(All cohomology groups in this paper have
rational coefficients.) More generally when $0$ is not a regular value
of $\mu$ we can consider the restriction map $H^*_K(M)\longrightarrow H^*_K(\mu^{-1}(0))$
and try to use the partial desingularization construction described above to
relate $H^*_K(\mu^{-1}(0))$ to the intersection cohomology of the hyperk\"{a}hler quotient
$\mu^{-1}(0)/K$ (cf. Remark 4.9 below).
In the setting of symplectic quotients the analogous map $\kappa$ is
surjective (\cite{Kir} Theorem 5.4), at least when the moment map is
proper, so a natural question is
whether or not $\kappa$ is surjective
 in the hyperk\"{a}hler setting. In the preprint \cite{k6}, the
third named author attempted to study this by modifying the ideas
and techniques in \cite{Kir}, but there is a crucial sign error underlying
the argument of \cite{k6} Lemma 4.2 which invalidates the approach of
\cite[\S4]{k6}. In this paper, we follow the approach in \S5 of
\cite{k6} and attempt to prove surjectivity in two stages:
\begin{enumerate}
\item surjectivity of the restriction $H^*_K(M)\to H^*_K(W)$,
where $W=\mu_\cc^{-1}(0)$,
by using equivariant Morse theory with respect to $-\|\mu_\cc\|^2$,
\item surjectivity of the restriction $H^*_K(W)\to
H^*_K(\mu^{-1}(0))$ by using equivariant Morse theory with respect to
$-\|\mu_1\|^2$.\end{enumerate}
Here the norm is induced from a fixed $K$-invariant inner product on
the Lie algebra $\mathfrak{k}$ of $K$, which we will use to identify
$\mathfrak{k}$ with its dual throughout.

To deal with the first stage, we reproduce the following theorem of
the third named author from \cite[\S5]{k6}.
\begin{theorem} \label{th1.2-new}
The Morse stratification of the function $\|\mu_\cc\|^2$ is
$K$-equivariantly perfect if $-\|\mu_\cc\|^2 $ is flow-closed.
\end{theorem}
Here, a function $f$ is \emph{flow-closed} if the gradient flow of
$f$ from any point is contained in a compact set. Because
\cite{k6} contains errors and is unpublished, we provide a full
corrected proof of Theorem \ref{th1.2-new} in this paper.

There still remain two major difficulties in this approach. The
first is the convergence issue of the gradient flows of the norm
squares of the moment maps (that is, whether the norm squares are {
flow-closed}),
and the second is that $W$ is in general singular. We avoid the
second difficulty by using a generic rotation of the \hk frame of
$\ii,\jj,\kk$. The three complex structures $\ii,\jj,\kk$ give us in
fact an $S^2$-family of complex structures and the triple
$(\ii,\jj,\kk)$ is nothing but a choice of an orthonormal frame. We
show that $W=\mu_\cc^{-1}(0)$ is smooth if we use a frame which is
general in the sense of Definition \ref{defn4.3} and $0$ is a
regular value of $\mu$; moreover, even when $0$ is not a regular
value of $\mu$, for a general frame there are no non-minimal
critical points of $\|\mu_1\|^2$ in $W$. Together with Theorem
\ref{th1.2-new}, this enables us to formulate a criterion for
surjectivity of $\kappa$.
\begin{theorem}
For a general choice of hyperk\"ahler frame
$\{\mathbf{i},\mathbf{j},\mathbf{k}\}$, 
if
\begin{enumerate}\item
$-\|\mu_\cc\|^2$ on $M$ is flow-closed, and
\item $-\|\mu_1\|^2$ on $W$ is flow-closed
\end{enumerate} then the restriction map
$$H^*_K(M)\to H^*_K(\mu^{-1}(0))$$ 
is surjective.\end{theorem}

When $M$ is a quaternionic vector space $V=\hh^n$ and the action is
linear, the flow-closedness of $-||\mu_1||^2$ follows from work of
Sjamaar \cite{Sja-Conv}. So we obtain
\begin{proposition} \label{prop:linear} Let $K$ act linearly on $V=\hh^n$, and let $\mu_\cc$ be
the holomorphic moment map corresponding to a general
choice of \hk frame $\{\ii,\jj,\kk\}$.
If $-\|\mu_\cc\|^2$ on $V$ is flow-closed, then 
the restriction map
$$H^*_K(V)=H^*_K\longrightarrow H^*_K(\mu^{-1}(0)) 
$$ is surjective. Here $H^*_K = H^*(BK)$ denotes the equivariant cohomology of a point.
\end{proposition}
Therefore, in the situation of Proposition \ref{prop:linear}
the surjectivity of $\kappa$ 
 follows from the
following.
\begin{conjecture} \label{conj:linear}
For a Hamiltonian linear \hk action
of a compact Lie group $K$
on a quaternionic vector space, $-\|\mu_\cc\|^2$ is flow-closed.
\end{conjecture}

Although it is purely a calculus problem in this linear case, the
question of flow-closedness seems to be difficult. But we prove the
following.
\begin{proposition} \label{p:flowclosed}
Conjecture \ref{conj:linear} is true when $K$ is the circle group $U(1)$.
\end{proposition}
\noindent For certain other  torus actions on quaternionic
vector spaces, flow-closedness follows immediately
from the $U(1)$ case (see Remark \ref{rk5.3}).
Note that when $K $ is a torus, surjectivity of $\kappa$ follows from the
work of Konno \cite{konno}, where the cohomology ring of the
hyperk\"ahler quotient is computed explicitly.

The layout of this paper is as follows. In \S2 we review general
results on hyperk\"ahler quotients. In \S3 we explain how to
construct a partial desingularisation of a hyperk\"ahler quotient
by a series of blow-ups. In \S4 we outline criteria for the
surjectivity of the map $\kappa$ from the equivariant cohomology
of a hyperk\"ahler manifold to the ordinary cohomology of the
hyperk\"ahler quotient. In  \S\ref{s:oldapp}  we reproduce the
argument of \cite[\S5]{k6} which shows that $-\|\mu_\cc\|^2$ is
equivariantly perfect if it is flow-closed. In \S6 we apply this
to hyperk\"ahler quotients of linear actions on quaternionic
vector spaces preserving the hyperk\"ahler structure.

\section{Hyperk\"ahler quotients}
In this section, we recall basic definitions and facts about
hyperk\"ahler quotients.
\begin{definition} A \emph{\hk manifold} is a Riemannian manifold
$(M,g)$ equipped with three complex structures $\ii$, $\jj$, $\kk$
such that \begin{enumerate} \item $\ii \jj=\kk=-\jj \ii$, \item
$\ii$, $\jj$, $\kk$ are orthogonal with respect to $g$, \item
$\ii$, $\jj$, $\kk$ are parallel with respect to the Levi-Civita
connection of $g$.\end{enumerate}
\end{definition}

\begin{remit}\begin{enumerate}
\item \emph{$T_pM\cong\hh^n$ and so $\dim_\rr M=4n$ for some
$n\in\zz$.} \item \emph{There exists an $S^2$-family of complex
structures on $M$. Indeed, for any $(a,b,c)\in \rr^3$ with
$a^2+b^2+c^2=1$, it is easy to see that $a\ii+b\jj+c\kk$ is a complex structure on $M$.
There exists a complex manifold $Z$ (the \lq twistor space' of $M$)
and a holomorphic map $\pi:Z\to
\pp^1\cong S^2$ such that the fiber over $(a,b,c)$ is the complex
manifold $(M,a\ii+b\jj+c\kk)$. 
}
 \item \emph{There exists an
$S^2$-family of (real) symplectic forms on $M$. Indeed, for each
complex structure $I$ in the $S^2$-family described above, $\omega_I(-,-)=g(-,I-)$
defines a symplectic form on $M$ which is K\"{a}hler with respect to $I$. We let $\omega_1,
\omega_2,\omega_3$ denote the symplectic forms defined by $\ii,
\jj,\kk$ respectively.}
\end{enumerate}
\end{remit}
\begin{remit} \emph{The complex valued two-form
$\omega_\cc=\omega_2+\sqrt{-1}\omega_3$ is holomorphic symplectic,
so a hyperk\"ahler manifold is always a {holomorphic symplectic
manifold}. A holomorphic symplectic manifold which is compact and
K\"ahler always admits a hyperk\"ahler metric, but this is not true
in general for non-compact $M$. }\end{remit}
\begin{example}
\emph{Examples of \hk manifolds include}
\begin{enumerate}
\item \emph{$\hh^n=T^*\cc^n$ and quotients of $\hh^n$ by discrete
group actions;} \item \emph{K3 surfaces, Hilbert schemes of points on a
K3 surface and generalized Kummer varieties.}
\end{enumerate}
\emph{When $X$ is a K\"ahler
manifold, the cotangent bundle of $X$ is holomorphic symplectic
but is not necessarily hyperk\"{a}hler.}
\end{example}

Let $K$ be a compact Lie group acting on a \hk manifold $M$.
Suppose that this action preserves the \hk structure $(g,\ii,
\jj,\kk)$.

\begin{definition} A \emph{moment map} for the
$K$-action on $M$ is a differentiable map
\[
\mu=(\mu_1,\mu_2,\mu_3):M\longrightarrow \mathfrak{k}^*\otimes_\rr
\rr^3
\]
satisfying \begin{enumerate} \item $\mu$ is $K$-equivariant, i.e.
$\mu(kx)=Ad_k^*\mu(x)$ for $k\in K$, $x\in X$; \item $\langle
\mathrm{d}\mu(v),\xi\rangle
=(\omega_1(\xi_m,v),\omega_2(\xi_m,v),\omega_3(\xi_m,v))$ for
$m\in M$, $v\in T_mM$ and $\xi = (\xi_1,\xi_2,\xi_3) \in \mathfrak{k}^*\otimes_\rr
\rr^3$, where $$\langle
\mathrm{d}\mu(v),\xi\rangle = (\langle
\mathrm{d}\mu_1(v),\xi_1\rangle,\langle
\mathrm{d}\mu_2(v),\xi_2\rangle,\langle
\mathrm{d}\mu_3(v),\xi_3\rangle).$$
\end{enumerate}

We say the action of $K$ is \emph{Hamiltonian} if a moment map
exists. We call $\mu_1$ the \emph{real moment map} and
$\mu_\cc:=\mu_2+\sqrt{-1}\mu_3$ the \emph{complex moment map}
or \emph{holomorphic moment map} with
respect to the frame $\{\ii,\jj,\kk\}$.
\end{definition}
\begin{remit}\emph{
(1) If $\mu$ is a moment map then any translation of $\mu$ by
central elements of $\mathfrak{k}^*$ is again a moment map.}

\noindent \emph{(2) It is straightforward to see that
$\mu_\cc:M\to \mathfrak{k}^*\otimes_\rr \cc$ is a holomorphic
function with respect to $\ii$. }
\end{remit}

\begin{theorem} \emph{(\cite{Hit1})} The smooth part of
$M/\!/\!/K:=\mu^{-1}(0)/K$ inherits a \hk structure from $M$.
\end{theorem}

\begin{example}
\begin{enumerate}
\item \emph{Suppose that $M=T^*\cc^n=\cc^{2n}$ and that the
action of $K=U(1)$ is given by
$\lambda(x,y)=(\lambda x,\lambda^{-1}y)$ for $x,y\in \cc^n$. Then
up to the addition of constants we have $\mu_1(x,y)=\frac{\sqrt{-1}}2(x^\dagger x-y^\dagger y)$ and
$\mu_\cc(x,y)=x^Ty$. Hence $W=\mu_\cc^{-1}(0)$ is an affine quadric
hypersurface in $\cc^{2n}$ and the variation of
$\mu^{-1}(c,0,0)/U(1)$ around $c=0$ is the} Mukai flop. \item
\emph{Suppose
$M=T^*\mathrm{Hom}(\cc^k,\cc^n)=\mathrm{Hom}(\cc^k,\cc^n)\times
\mathrm{Hom}(\cc^n,\cc^k)$ and $K=U(n)$ acts on $M$ in the
natural way. Then the action is Hamiltonian with $
\mu_1(x,y)=\frac{\sqrt{-1}}2 (xx^\dagger-y^\dagger y)$ and
$\mu_\cc(x,y)=xy$.} \item \emph{Suppose
$M=T^*\mathrm{End}(\cc^n)=\mathrm{End}(\cc^n)\times
\mathrm{End}(\cc^n)$ and $K=U(n)$ acting by conjugation. Then up
to the addition of central constants we have
$\mu_1(B_1,B_2)=\frac{\sqrt{-1}}2([B_1,B_1^\dagger]+[B_2,B_2^\dagger])$
and $\mu_\cc(B_1,B_2)=[B_1,B_2]$. } \item \emph{Suppose $K$ acts on
$M$ with moment map $\mu:M\to \mathfrak{k}^*\otimes \rr^3$. Let
$\rho:H\to K$ be a homomorphism and let $\mathfrak{h}$ be the Lie
algebra of $H$. A moment map for the induced action of $H$ on $M$ is
the composition $M\to \mathfrak{k}^*\otimes \rr^3\to
\mathfrak{h}^*\otimes \rr^3$ of $\mu$ and the dual
$\mathfrak{k}^*\to \mathfrak{h}^*$ of the tangent map of $\rho$. }
\item \emph{Let $M_1,M_2$ be two \hk manifolds acted on by $K$ in
Hamiltonian fashion. Then the sum of the moment maps for $M_1$ and
$M_2$ is a moment map for the diagonal action on the product
$M_1\times M_2$.} \item \emph{(ADHM spaces) Let
$M=T^*\mathrm{End}\cc^n\times T^*\mathrm{Hom}(\cc^k,\cc^n)$ and
$K=U(n)$. Then
$$\mu_1(B_1,B_2,x,y)=\frac{\sqrt{-1}}2([B_1,B_1^\dagger]+[B_2,B_2^\dagger]
+xx^\dagger-y^\dagger y)$$ and $$\mu_\cc(B_1,B_2,x,y)=[B_1,B_2]+xy.$$
The quiver variety construction is similar \cite{ADHM,quiver}.  }
\item
\emph{(coadjoint orbits) If $O$ is an orbit in
$\mathfrak{g}^*=\mathfrak{k}^*\otimes_\rr\cc$, then
$\omega_\alpha(\xi_\alpha,\eta_\alpha)=\langle
\alpha,[\xi,\eta]\rangle$ defines a holomorphic symplectic form on
$O$ and $\mu_\cc:O\hookrightarrow \mathfrak{g}^*$ is the complex
moment map (cf. \cite{Kron, Biqu}).}
\end{enumerate}
\end{example}

%
%
%
%

\section{Partial desingularization}
Our first result in this paper allows us to construct a partial
resolution of singularities of a hyperk\"ahler quotient through a
sequence of explicit blow-ups, which is closely related to the
partial desingularization construction in \cite{Kir2}.

Let $\mu:M\to \mathfrak{k}^*\otimes \rr^3$ be a moment map for a
Hamiltonian $K$-action on a hyperk\"ahler manifold $M$.
Let $\pi:\mu^{-1}(0)\to \mu^{-1}(0)/K=M/\!/\!/K$ be the quotient
map. If $x\in \mu^{-1}(0)$ has trivial stabilizer, then $\pi(x)$
is a smooth point of $M/\!/\!/K$ \cite{Hit1}. 
We fix a
preferred complex structure $\mathbf{i}$ so that
$$\mu=(\mu_1,\mu_\cc):M\to \mathfrak{k}^*\otimes (\rr\oplus \cc)$$
and view $M$ as a complex manifold with respect to $\mathbf{i}$,
equipped with a holomorphic symplectic two-form
$\omega_\cc=\omega_2+\sqrt{-1}\omega_3$.
The function
$$ \mu_\cc=\mu_2 + i \mu_3: M \to \mathfrak{k}^* \otimes \cc $$
is holomorphic with respect to the complex structure $\ii$ on $M$,
and thus $W = {\mu_2}^{-1}(0) \cap {\mu_3}^{-1}(0)$ is a complex
analytic subvariety of $M$. Moreover $\mu^{-1}(0)/K = W \cap
{\mu_1}^{-1}(0)/K$ is the K\"ahler quotient of $W$ by the action of
$K$.

We can apply the partial desingularization construction of
\cite{Kir2} to the K\"{a}hler quotient ${\mu_1}^{-1}(0)/K$ of $M$
by $K$ with respect to the K\"{a}hler structure corresponding to
$\ii$. This gives us a complex orbifold with a holomorphic
surjection to ${\mu_1}^{-1}(0)/K$ which restricts to an
isomorphism over the open subset of ${\mu_1}^{-1}(0)/K$ where the
derivative of $\mu_1$ is surjective. The proper transform
$\widetilde{W/\!/K}$ of   $W/\!/K = W \cap {\mu_1}^{-1}(0)/K =
\mu^{-1}(0)/K$ in ${\mu_1}^{-1}(0)/K$ is the result of applying
the partial desingularization construction of \cite{Kir2} to the
complex subvariety $W$ of $M$ instead of $M$ itself, and thus if
$W$ were nonsingular then it would follow immediately that
$\widetilde{W/\!/K}$ is a complex orbifold. However in general $W$
is singular (indeed, if $W$ is nonsingular then 0 is a regular
value of $\mu$ and so no partial desingularization for $M/\!/\!/K
= \mu^{-1}(0)/K$ is needed). Nonetheless, we will see in this
section that $\widetilde{W/\!/K}$ is always a complex orbifold:
the blow-ups in its construction resolve the singularities coming
from $W$ at the same time as reducing the singularities created by
the quotient construction to orbifold singularities.

The construction of $\widetilde{W/\!/K}$ given in \cite{Kir2} is as
follows. Let $H_0$ be the identity component of the stabilizer of a
point in $\mu^{-1}(0) = W \cap {\mu_1}^{-1}(0)$ whose dimension is
maximal possible among such, and let $Z_{H_0}$ be the fixed point
set of $H_0$ in $W$. First we blow $W^{ss}$ up along
$GZ_{H_0}^{ss}$, and give the resulting blow-up a K\"{a}hler
structure which is a small perturbation of the pull-back of the
K\"{a}hler structure on $M$ restricted to $W$. Here $W^{ss}$ denote
the open subset of semistable points $x$ in $W$, i.e. the
$K^\cc$-orbit closure of $x$ meets $\mu_1^{-1}(0)$. Then we repeat
this process until, after finitely many steps, the points in the
blow-up where the real moment map vanishes all have 0-dimensional
stabilizers, and the resulting quotient is $\widetilde{W/\!/K}$.
Equivalently, by \cite[Lemma 3.11]{Kir2}, we can construct
$\widetilde{W/\!/K}$ directly from the quotient ${W/\!/K}$ by a
sequence of blow-ups.

\begin{theorem}\label{thm3.1} Suppose there is at least one point
in $\mu^{-1}(0)$ whose stabilizer is finite. The above blow-up
procedure produces a surjective analytic map
$$\widetilde{W/\!/K}\to W/\!/K$$ such that
\begin{enumerate} \item it is an isomorphism on the open subset of
the orbits of $x\in \mu^{-1}(0)$ where $d\mu_1$ is
surjective;\item $\widetilde{W/\!/K}$ has at worst orbifold
singularities.\end{enumerate}
\end{theorem}
\begin{definition}
We call $\widetilde{W/\!/K}$ the \emph{partial desingularization}
of the \hk quotient $M/\!/\!/K$ with respect to $\mathbf{i}$.
\end{definition}

Item (1) is obvious since the blow-ups do not touch the orbits of
$x\in\mu^{-1}_1(0)$ where $d\mu_1$ is surjective. In order to
prove that $\widetilde{W/\!/K}$ has only orbifold singularities
even though $W$ is in general singular, we shall prove that the
singularities of $W$ are all quadratic (Theorem \ref{theorem3.2}).
Then we can deduce from the following lemma that after blowing up
$W$ as described in the previous paragraphs, we obtain a variety
$\widetilde{W}$ which is nonsingular along the proper transform of
$\mu^{-1}(0)$. It then follows immediately from \cite{Kir2} that
the quotient $\widetilde{W/\!/K}$ of $\widetilde{W}$ has at worst
orbifold singularities.
\begin{lemma}
Let $V\subset \cc^n$ be a complex analytic variety defined only by
\emph{homogeneous} quadratic polynomials. Let $\pi:\tilde{V}\to V$
be the blow-up at $0$. Then $\tilde{V}$ is isomorphic to a line
bundle $L=\cO_E(2)$ over the exceptional divisor $E=\pi^{-1}(0)$. In
particular, $E$ is normally nonsingular in $\tilde{V}$.
\end{lemma}
\begin{proof}
Let $A_1,\cdots, A_r$ be symmetric matrices such that the
quadratic forms $f_i(x)=x^tA_ix$, for $i=1,\cdots,r$, define $V$.
Let $p:\widetilde{\cc^n}\to \cc^n$ be the blow-up at $0$ and
consider the local chart
\[
(y_1,\cdots, y_n)\mapsto (y_1,y_1y_2,\cdots,
y_1y_r)=(x_1,x_2,\cdots,x_n).
\]
Then the blow-up $\tilde{V}$ is the subvariety of
$\widetilde{\cc^n}$ defined by
\[
\tilde{f}_i(y)=\frac{1}{y_1^2}f_i(x)=\left(\begin{matrix} 1 &y_2&
\cdots & y_n\end{matrix}\right) A_i\left(\begin{matrix} 1 \\ y_2\\
\cdots \\ y_n\end{matrix}\right).
\]
Since all the defining equations for $\tilde V$ are independent of
$y_1$, $\tilde V\cap U_1\cong \cc\times (E\cap U_1)$ where
$U_1=\cc^n=\{(y_1,\cdots, y_n)\}\subset \widetilde{\cc^n}$. Here $E$
is the projective variety in $\pp^{n-1}=\pp \cc^n$ defined by the
homogeneous polynomials $f_1,\cdots, f_r$. It is an easy exercise to
check that these local trivializations for $\tilde V$ glue to give
us the line bundle $L=\cO_E(2)$, which is the restriction of
$\cO_{\pp^{n-1}}(2)$ to $E$.
\end{proof}
\begin{remit}\label{remindres} \emph{ Suppose $$V_k^0\longrightarrow \cdots\longrightarrow
V_1^0\longrightarrow V^0_0=V^0
$$ is a sequence of blow-ups, which resolves the singularities of
$V^0=V-\{0\}$, with blow-up centers $C_i^0$ satisfying
\begin{enumerate}
\item $C_i^0$ is nonsingular; \item $C_i^0\subset V_{i-1}^0$ is the
proper transform of its image $B_i^0$ in $V^0$; \item $B_1^0\subset
B_2^0\subset\cdots \subset B_k^0$; \item each $B_i^0$ is conic in
the sense that the closure $B_i$ of $B_i^0$ in $\tilde{V}$ is the
restriction of the line bundle $L$ to $B_i\cap E$.\end{enumerate}
Let
$$\tilde{V}_k\longrightarrow \cdots\longrightarrow
\tilde{V}_1\longrightarrow \tilde{V}
$$ be the blow-ups along the proper transforms $C_i\subset \tilde{V}_{i-1}$ of the
closures $B_i\subset \tilde{V}$ of $B_i^0$. Since the blow-up
centers are all normally nonsingular along $E$ by (4) above, the
proper transform of $E$ is normally nonsingular in $\tilde{V}_k$
and hence $\tilde{V}_k$ is nonsingular. }\end{remit}

To show that the singularities of $W$ are all quadratic, we begin
with a holomorphic version of the equivariant Darboux theorem.
\begin{theorem}\label{equiDar}
Suppose $\omega_0, \omega_1$ are holomorphic symplectic two-forms
on a complex manifold $X$ and $Y$ is a complex submanifold of $X$
such that $\omega_0$ and $\omega_1$ coincide on $TX|_Y$. Suppose
further that there exists an open neighborhood $U$ of $Y$ and a
differentiable family of holomorphic maps $\varphi_t:U\to U$ for
$0\le t\le 1$ such that $$\varphi_1=\mathrm{id}_U,\quad
\varphi_0(U)=Y, \quad \varphi_t|_Y=\mathrm{id}_Y \ \ \forall t.$$
Then there exists an open neighborhood $U'$ contained in $U$ and a
biholomorphic map $f:U'\to f(U')$ where $f(U')$ is an open subset of $X$ such that
\[
f|_Y=\mathrm{id}_Y\quad \text{and} \quad f^*\omega_1=\omega_0.
\]
If a compact Lie group $K$ acts on $X$ with $\omega_0$,
$\omega_1$, $Y$, $U$ and $\varphi_t$ all $K$-equivariant, then $f$
can be chosen to commute with the action of $K$.
\end{theorem}
\begin{proof} The proof is an obvious modification of Weinstein's
in \cite{GuillStern}. Let $\omega_t=\omega_0+t\sigma$ where
$\sigma=\omega_1-\omega_0$. By assumption, $\sigma$ is trivial on
$Y$ and $\mathrm{d}\sigma=0$. Let $\xi_t$ denote the vector field
of $\varphi_t$. Then
\[
\sigma - \varphi_0^*\sigma
=\int_0^1\frac{\mathrm{d}}{\mathrm{d}t}(\varphi_t^*\sigma)\mathrm{d}t
=\int_0^1\varphi_t^*(\imath(\xi_t)\mathrm{d}\sigma)\mathrm{d}t
+\mathrm{d}\int_0^1\varphi_t^*(\imath(\xi_t)\sigma)\mathrm{d}t=\mathrm{d}\beta
\]
where $\beta=\int_0^1\varphi_t^*(\imath(\xi_t)\sigma)\mathrm{d}t$.
By our assumption, $\beta|_Y=0$. We define a \emph{holomorphic}
vector field $\eta_t$ on $U$ by the equation
$\imath(\eta_t)\omega_t=-\beta$. By shrinking $U$ to a smaller
open neighborhood $U'$ if necessary, we can integrate $\eta_t$ to
obtain \emph{biholomorphic} maps $f_t:U'\to X$ for $0\le t\le 1$.
Then \[
f_1^*\omega_1-\omega_0=\int_0^1\frac{\mathrm{d}}{\mathrm{d}t}(f_t^*\omega_t)
\mathrm{d}t=\int_0^1f_t^*(\sigma+\mathrm{d}\imath(\eta_t)\omega_t)\mathrm{d}t=0
\]
as desired. Moreover if  $K$ acts on $X$ with $\omega_0$,
$\omega_1$, $Y$, $U$ and $\varphi_t$ all $K$-equivariant, then
$f_1$ is also $K$-equivariant.
\end{proof}

Let $x\in \mu^{-1}(0)$, and note that if
$\alpha,\beta,\gamma,\delta \in \mathfrak{k}$ then $\alpha_x,\ii
\beta_x,\jj \gamma_x$ and $\kk \delta_x$ are mutually orthogonal,
since, for example,
$$g(\jj \gamma_x,\kk \delta_x) = g(\kk \ii \gamma_x,\kk \delta_x)
=  g( \ii \gamma_x, \delta_x) = d\mu_1(\delta_x)\cdot\gamma
=ad^*_\delta(\mu_1(x))\cdot\gamma = 0$$ because $\mu_1(x)=0$. We
next find a local model for a neighborhood of $x$ in $M$. Let $H$
be the stabilizer of $x$ in $K$, so that the complexification
$H^\cc$ is the stabilizer of $x$ in the complexification
$G=K^\cc$, and let $Y=Gx\cong G/H^\cc$ be its $G$-orbit. Further,
let $B$ be a ball in $T_x(Gx)^\perp=\mathbf{j}(V_x^\cc)\oplus
\mathcal{W}$ where $V_x=T_x(Kx)$ and $$\mathcal{W}=(V_x \oplus \ii
V_x \oplus \jj V_x \oplus \kk V_x)^\perp.$$ Note that
$\mathcal{W}$ is a quaternionic vector space with a linear action
of $H^\cc$. Let $S=H^\cc B\subset T_xM$. Then the holomorphic
slice theorem in \cite{Sjamaar} tells us that
\[
G\times_{H^\cc}S
\]
is biholomorphic to an open neighborhood $U$ of $Y$ in $M$. Now
consider
\[
X=G\times_{H^\cc}(\mathbf{j}V_x^\cc\oplus
\mathcal{W})
\]
which is a hyperk\"ahler manifold since it is a hyperk\"ahler
quotient of $T^*G\times \mathcal{W}$ by the action of $H$. By this
quotient construction, we obtain a holomorphic symplectic form
$\omega_\cc'$ and a complex moment map
\begin{equation} \label{cmm}
\mu'_\cc(g,\mathbf{j}b,w)=Ad_g^*(b+\mu_\cc^{\mathcal{W}}(w))\qquad
g\in G, b\in V_x^\cc, w\in \mathcal{W}
\end{equation} for the action of $K$ on
$X$ where $\mu_\cc^{\mathcal{W}}$ is the complex moment map for the
action of $H$ on $\mathcal{W}$. If we take the hyperk\"ahler
quotient of $X$ by the action of $K$, then we obtain the
hyperk\"ahler quotient $\mathcal{W}/\!/\!/H$. Since $X\supset U$,
the pull-back of the holomorphic symplectic form $\omega_\cc$ of $M$
is also holomorphic symplectic, and this coincides with $\omega'_\cc$
on $Y$. On the other hand, $X$ is a vector bundle over $Y=G/H^\cc$
with fibers $ \mathbf{j}V_x^\cc\oplus
\mathcal{W} $  and
hence by shrinking the fibers, we obtain a family of holomorphic
maps $\varphi_t:X\to X$ satisfying the assumptions of Theorem
\ref{equiDar}. Thus by Theorem \ref{equiDar}, we obtain

\begin{theorem} \label{theorem3.2}
There exists a biholomorphic map from a neighborhood of $G/H^\cc$
in $X$ to a neighborhood of the orbit $Gx$ in $M$ such that the
pull-backs of $\omega_\cc$ and $\mu_\cc$ are $\omega'_\cc$ and
$\mu'_\cc$ as at (\ref{cmm}) above. Furthermore, in a
$K$-invariant neighborhood of $x$, the embedding
$W=\mu^{-1}_\cc(0)\hookrightarrow M$ is biholomorphic to
\[
G\times_{H^\cc}(0\times
(\mu_\cc^{\mathcal{W}})^{-1}(0))\hookrightarrow G\times
_{H^\cc}(\mathbf{j}V_x^\cc\oplus \mathcal{W}).
\]
\end{theorem}
In particular, the singularity of $M/\!/\!/K$ at $\pi(x)$ is
precisely the hyperk\"ahler quotient $\mathcal{W}/\!/\!/H$ of the
vector space $\mathcal{W}$.

In the simplest case when $K=U(1)$ is the circle group, then it follows
immediately from this theorem that the blow-up of $W$ along the fixed point
set of $K$ is nonsingular, and hence $\widetilde{W/\!/K}$ has only orbifold
singularities.

\begin{corollary} Suppose that
$K = U(1) =H$.
Then the blow-up of $W=\mu_\cc^{-1}(0)$ along $Z_{H}$ where $Z_{H}$
is the fixed point set of $H$ in $W$ is nonsingular, and $\widetilde{W/\!/K}$ has only orbifold
singularities.
\end{corollary}
\begin{proof} The \hk moment map $\mu$ is locally constant on the $K$-fixed
point set in $M$, so that $Z_{H}$ is a union of some of its connected components
and hence is nonsingular.
In the local model $X$, we can identify $GZ_{H}= Z_{H}$ with $0\times
\mathcal{W}^{H}$ where $\mathcal{W}^{H}$ is the subspace of
vectors fixed by $H$. Also $W$ in this local model is $0\times
(\mu_\cc^{\mathcal{W}})^{-1}(0)$. Since $\mu_\cc^\mathcal{W}$ is
homogeneous quadratic, $(\mu_\cc^{\mathcal{W}})^{-1}(0)$ is the
product of an affine quadric cone and $\mathcal{W}^{H}$ and so
after the blow-up along $Z_{H}$ we obtain a complex manifold
whose K\"{a}hler quotient $\widetilde{W/\!/K}$ has only orbifold
singularities.
\end{proof}

In general for any compact group $K$, we can describe the
singularities of $W$ along the blow-up center of the first blow-up
in the partial desingularization process as follows. Let $H_0$ be
the identity component of the stabilizer of a point in
$\mu^{-1}(0)$, whose dimension is maximal possible among such. Then
$GZ_{H_0}^{ss}$, where $Z_{H_0}$ is the fixed point set of $H_0$ in
$W$, is closed and nonsingular in a neighborhood of $\mu^{-1}(0)$.
This follows from \cite{Kir2} Corollary 5.10 and Lemma 5.11 for the
action of $G$ on $M$, together with the local model of Theorem
\ref{theorem3.2} and the fact that the \hk moment map for the action
of $H_0$ is locally constant on $Z_{H_0}$.

By Theorem \ref{theorem3.2}, we can locally identify $GZ_{H_0}^{ss}$
with $G \times_{H^\cc} (0\times \mathcal{W}^{H_0})$ where
$\mathcal{W}^{H_0}$ is the subspace of vectors fixed by $H_0$. Also
$W$ in this local model is $G\times _{H^\cc}(0\times
(\mu_\cc^{\mathcal{W}})^{-1}(0))$. Since $\mu_\cc^\mathcal{W}$ is
homogeneous quadratic, $(\mu_\cc^{\mathcal{W}})^{-1}(0)$ is the
product of an affine quadric cone and $\mathcal{W}^{H_0}$ and so
after the blow-up along $GZ_{H_0}^{ss}$ we obtain a line bundle
along the exceptional divisor.

By induction on the dimension of $H_0$, suppose we could resolve the
singularities of the complement of $GZ_{H_0}^{ss}$ by the partial
desingularization process. Then we apply Remark \ref{remindres} to
conclude that $W$ becomes smooth after the blow-ups in the partial
desingularization process of $M/\!/K$. Thus the proof of Theorem
\ref{thm3.1} is completed.

%

\begin{remit} \emph{ 
Note however that the construction of this partial desingularization
$\widetilde{M/\!/\!/K}$ depends on
the choice of preferred complex structure $\ii$, and $\widetilde{M/\!/\!/K}$
does not
inherit a \hk structure from that of $M$.}
\end{remit}

\begin{remit} \emph{
Another way to (partially) resolve the singularities of a hyperk\"ahler quotient
is to perturb the moment map by a small central element in
$\mathfrak{k}^*$. This is better in the sense that it gives us a partial
desingularization which is hyperk\"ahler again. However this resolution is
possible only when the center contains regular values of the
\hk moment map and it does not apply to
some examples, such as the moduli of Higgs bundles, or
$\mathbb{H}^n\otimes sl(2)/\!/\!/SL(2)$. In fact, there are no \hk
resolutions for the latter singularity \cite{KLS,KY}.
}
\end{remit}

\begin{remit}\emph{
One can use the partial desingularization process to calculate the
intersection cohomology of singular \hk varieties, such as the
moduli space of rank 2 semistable Higgs bundles with trivial
determinant over a Riemann surface. See \cite{KY2} for
details.}\end{remit}

%
%
%
%

\section{Surjectivity criterion} \label{s:s4}
In this section, we give a criterion for the surjectivity of the
hyperk\"ahler Kirwan map $H^*_K(M)\to H^*(M/\!/\!/K)$.

Let $M$ be a \hk manifold\footnote{A smooth manifold is a second
countable Hausdorff space equipped with an atlas. It has at most
countably many connected components.} on which a compact connected
Lie group $K$ acts preserving the \hk structure. Suppose we have a
\hk moment map $\mu=(\mu_1,\mu_\cc)$ which takes values in
$\mathfrak{k}^*\otimes (\rr\oplus \cc)=\mathfrak{k}^*\otimes
\rr^3$ for the action of $K$. Let $T$ be a maximal torus of $K$.
The following is a simple consequence of Mostow's theorem
\cite{Mostow}.
\begin{lemma} Consider the action of $T$ on $M$ and the
stabilizers of points in $M$. Let $\mathcal{T}$ be the set of
nontrivial subtori $T'$ of $T$ which are the identity components
of the stabilizers of points in $M$. Then $\mathcal{T}$ is
countable.
\end{lemma}

As a consequence, the set of connected components of the fixed
point sets of $T'$ for $T'\in \mathcal{T}$ is countable and we
write it as $\{Z_j\,|\, j\in \zz_{\ge 0}\}$. For each $Z_j$ let
$T_j$ be the identity component of the stabilizer in $T$ of a
general point in $Z_j$.\footnote{This makes sense because all the
orbit type strata have even real codimension.}
\begin{lemma} For any $\gamma\in \mathrm{Lie}(T_j)$, $\langle
\mu (Z_j),\gamma\rangle $ consists of a single point in $\rr^3$.
\end{lemma}
The proof of this lemma is elementary and we omit it.

Now for each $j$ we choose $\gamma_j\in \mathrm{Lie}(T_j)$ such
that $\langle \mu(Z_j),\gamma_j\rangle \ne 0$ whenever
$$\rho_j:=\mathrm{sup}\{\|\langle \mu(Z_j),
\gamma\rangle\|\,\,:\,\, \|\gamma\|=1, \gamma\in \mathrm{Lie}(T_j)
\}>0.$$ We require no condition for $\gamma_j$ when $\rho_j=0$, in
which case we have
$$\langle \mu(Z_j),\gamma\rangle =0\quad \text{ for any }\quad
\gamma\in \mathrm{Lie}(T_j).$$ Since $\langle
\mu(Z_j),\gamma_j\rangle\in \rr^3$ is nonzero whenever $\rho_j>0$, a
general choice of  hyperk\"ahler frame
$\{\mathbf{i},\mathbf{j},\mathbf{k}\}$ (being uncountably many)
satisfies the following condition:
\medskip

\noindent {\it $\langle \mu(Z_j),\gamma_j\rangle $ are not of the
form $(a,0,0)\in \rr^3$ with $a\ne 0$ for all $j$ with $\rho_j>0$,
i.e. }
\begin{equation} \label{condition}
\langle \mu_\cc(Z_j),\gamma_j\rangle \ne 0. \end{equation}
\medskip

\begin{definition} \label{defn4.3} We say a hyperk\"ahler frame
$\{\mathbf{i},\mathbf{j},\mathbf{k}\}$ is \emph{general} in the
group of frames $SO(3)$ if the above condition is satisfied.
\end{definition}

\begin{definition}
As in \cite{Kir} we let
$$W^{ss} = \{ x \in W \ | \ \mbox{the gradient flow from
 $x$ with respect to} -\|\mu_1\|^2$$
 $$ \mbox{ has a limit point in $\mu_1^{-1}(0)$} \}.$$
 We say that $x \in W$ is \emph{semistable} if $x \in W^{ss}$.
 \end{definition}
%

\begin{proposition}\label{PropNoUnstable}
For a general hyperk\"ahler frame
$\{\mathbf{i},\mathbf{j},\mathbf{k}\}$, if the gradient flow from
each $x\in W=\mu_\cc^{-1}(0)$ with respect to $-\|\mu_1\|^2$ is
contained in a compact set, then $W=W^{ss}$. If furthermore
$\mu^{-1}(0)$ is smooth, $W=\mu_\cc^{-1}(0)$ is smooth.
\label{prop:above}
\end{proposition}

\begin{proof}
A point $x\in W$ is a critical point of $\|\mu_1\|^2 $ if and only
if $\mu_1(x)_x=0$, where $\mu_1(x)\in \mathfrak{k}^*$ is
identified with an element of $\mathfrak{k}$ by using the fixed
invariant inner product on $\mathfrak{k}$. (See \cite[Lemma
3.1]{Kir}.) Suppose $x\in W$ is a critical point of $\|\mu_1\|^2$
but $\mu_1(x)\ne 0$. Without loss of generality, replacing $x$
with $kx$ for some $k \in K$,
 we may assume $\mu_1(x)\in \mathrm{Lie}(T)$ and $x\in
Z_j$ for some $j$ with
$\mathrm{Lie}(\mathrm{Stab}_Tx)=\mathrm{Lie}(T_j)$. Let
$\gamma=\mu_1(x)\in \mathrm{Lie}(T_j)$. Then $\langle
\mu_1(Z_j),\gamma\rangle$ is a nonzero constant $\|\gamma\|^2$.
Hence $\rho_j>0$, so $\mu_\cc(x)\ne 0$ because $ \langle
\mu_\cc(x),\gamma_j\rangle = \langle \mu_\cc(Z_j),\gamma_j\rangle
\ne 0$ by (\ref{condition}). This is a contradiction to $x\in W$.
Therefore $\mu_1(x)=0$ whenever $x\in W$ is a critical point of
$\|\mu_1\|^2$ and hence all the points in $W$ are semistable. If
$\mu^{-1}(0)$ is smooth, all the stabilizers in $G = K^\cc$ of
points in $W$
are finite 
 and hence 
  $d\mu_\cc$ is surjective
and so $W$ is smooth.
\end{proof}

\begin{definition} \label{defn:flow-closed}
Let $f:M\to \rr$ be a smooth function defined on a manifold. We
say $f$ is \emph{flow-closed} if the gradient flow of $f$ from any
$x\in M$ is contained in a compact set.
\end{definition}

\begin{corollary}\label{cor5.6} 
 For a general hyperk\"ahler frame
$\{\mathbf{i},\mathbf{j},\mathbf{k}\}$, if
$-\|\mu_1\|^2$ on $W=\mu_\cc^{-1}(0)$ and $-\|\mu_\cc\|^2$ on $M$
are flow-closed, then the restriction map $H^*_K(M)\to
H^*_K(\mu^{-1}(0))$ is surjective.\end{corollary}
\begin{proof} The corollary follows from Theorem
\ref{thm-equivperfect} below which says that $-\|\mu_\cc\|^2$ is
equivariantly perfect if it is flow-closed. This theorem is proved
in \cite[\S5]{k6}. Because \cite{k6} contains errors and is
unpublished, we have reproduced an edited version of this section
in Section \ref{s:oldapp}. 

By Theorem \ref{thm-equivperfect}, the restriction map
$H^*_K(M)\to H^*_K(W)$ is surjective, and by Proposition
\ref{prop:above} above we have $W=W^{ss}$ and $H^*_K(W)\cong
H^*_K(\mu^{-1}(0))$. Thus we obtain the surjectivity.
\end{proof}

We have proved

\begin{theorem}\label{ThSurjCrit} \emph{(surjectivity criterion)}
 For a general hyperk\"ahler frame
$\{\mathbf{i},\mathbf{j},\mathbf{k}\}$,  if
\begin{enumerate}
\item $-\|\mu_\cc\|^2$ on $M$ is flow-closed, and \item
$-\|\mu_1\|^2$ on $W=\mu_\cc^{-1}(0)$ is flow-closed
\end{enumerate} then the restriction map
$$H^*_K(M)\to H^*_K(\mu^{-1}(0))$$ is
surjective. If furthermore 0 is a regular value of
the \hk moment map $\mu$, then  $W=\mu_\cc^{-1}(0)$ is
smooth and the map
$$\kappa: H^*_K(M)\to H^*_K(\mu^{-1}(0))\cong H^*(\mu^{-1}(0)/K)$$ is
surjective.
\end{theorem}



\begin{remit} \emph{Suppose that
$\{\mathbf{i},\mathbf{j},\mathbf{k}\}$ is a  general hyperk\"ahler frame
satisfying conditions (1) and (2) of Theorem \ref{ThSurjCrit}, so that
 the restriction map
\begin{equation} \label{one} H^*_K(M)\to  H^*_K(W) = H_K^*(W^{ss}) \cong
H^*_K(\mu^{-1}(0))\end{equation}
is
surjective, but suppose that 0 is not a regular value of
the \hk moment map $\mu$. Then as in \S3 we can construct a partial
desingularization $\widetilde{M/\!/\!/K}=\tilde{W}^{ss}/G$ of the \hk quotient
$M/\!/\!/K$ with respect to the preferred complex structure $\ii$
(where $G=K^\cc$ is the complexification of $K$ and $\tilde{W}^{ss}$ is
an open subset of a blow-up of $W = \mu_\cc^{-1}(0)$),
and the blow-down map induces a map on $K$-equivariant cohomology
\begin{equation} \label{two}
H_K^*(W^{ss}) \to H^*_K(\tilde{W}^{ss}) \cong H^*(\tilde{W}^{ss}/G)
=H^*(\widetilde{M/\!/\!/K}). \end{equation}
The intersection cohomology $IH^*(M/\!/\!/K)$ of $M/\!/\!/K$ (with respect
to the middle perversity and rational coefficients) is a direct summand of
$H^*(\widetilde{M/\!/\!/K})$ by the decomposition theorem of Beilinson, Bernstein,
Deligne and Gabber \cite[6.2.5]{BBDG}. This direct decomposition is not
in general canonical, but
Woolf \cite{Woolf} shows that the partial desingularization construction
(and more generally any {\em $G$-stable resolution} in the sense of
\cite[\S3]{Woolf}) can be used to define a projection
\begin{equation} \label{projj} H^*(\widetilde{M/\!/\!/K}) \to IH^*(M/\!/\!/K),
\end{equation} and also
a projection
\begin{equation} \label{decomp} IH^*_K(W^{ss}) \to IH^*(W/\!/K) = IH^*(M/\!/\!/K) \end{equation}
such that the composition of (\ref{decomp}) with the canonical map
from $H^*_K(W^{ss})$ to $IH^*_K(W^{ss})$ agrees with the composition of
(\ref{projj}) with (\ref{two}).
The argument of \cite{Woolf}
uses the decomposition theorem at the level of the equivariant derived category
(see \cite[\S5]{BL}) as a decomposition of complexes of sheaves up to
quasi-isomorphism, and it can
be applied to $M$ and restricted to its complex subvariety $W = \mu_\cc^{-1}(0)$
to give a projection from the hypercohomology of the restriction to
$W^{ss}=W$ of ${\mathcal IC}^\bullet_K(M^{ss})$ (which is just $H^*_K(W^{ss})$
since $M$ is nonsingular) to the hypercohomology of the restriction to
$W/\!/K = M/\!/\!/K$ of ${\mathcal IC}^\bullet(M/\!/K)$.
We thus find that
 the partial desingularization construction gives us a map
$$H^*_K(M) \to IH^*(M/\!/\!/K)$$
which is the composition of (\ref{one}), (\ref{two}) and the projection (\ref{projj})
from $H^*(\widetilde{M/\!/\!/K})$ onto $IH^*(M/\!/\!/K)$,
and factorises through a surjection from $H^*_K(M)$
to the hypercohomology of the restriction to
$W/\!/K = M/\!/\!/K$ of ${\mathcal IC}^\bullet(M/\!/K)$.
} \end{remit}


\section{Morse flow of the norm square of the complex moment map} \label{s:oldapp}

This section is a corrected version of part of the unpublished
preprint \cite{k6} (mainly \cite{k6} \S5). 
The proof of Lemma \ref{l5.5} differs from the original
only because some minor errors (in particular an unimportant sign error) in \cite{k6} are corrected, and as a result
the matrix in (\ref{e:sign}) differs by one sign from the matrix in \cite{k6}
(first equation p. 35). The conclusion
of \S5 of \cite{k6} (that the norm square of the complex moment map
is an equivariantly perfect Morse function) remains valid after these modifications.
The proof of \cite{k6} Lemma 4.2 is, however, incorrect  (as originally pointed out by Simon Donaldson), and the conclusions
of \cite[\S4]{k6} are therefore invalid.




Consider the function
$f_{23}: M \to \rr$ defined by
$$ f_{23}(x) = \|\mu_2(x)\|^2 + \|\mu_3(x)\|^2=\|\mu_\cc(x)\|^2. $$
For all $x \in M$ we have
\begin{equation} \label{f23} \grad f_{23} (x) = 2\jj (\mu_2(x)_x - \ii   \mu_3(x)_x)
= 2 \kk (\mu_3(x)_x + \ii \mu_2(x)_x). \end{equation}

The purpose of this section is to prove the following.
\begin{theorem}\label{thm-equivperfect}
The Morse stratification of the function $f_{23}=\|\mu_\cc\|^2$ is
$K$-equivariantly perfect if $-f_{23}$ is flow-closed.
\end{theorem}



\begin{lemma} \label{l:wilkin} 
If $\grad f_{23}(x) = 0$ and $\beta_2 = \mu_2(x)$ and $\beta_3 =
\mu_3(x)$ then
 $$[\beta_2, \beta_3] = 0 $$
\end{lemma}
\begin{proof}
Recall that $$ \mu_\cc=\mu_2 + \sqrt{-1} \mu_3: M \to
\mathfrak{k}^* \otimes \cc $$ is $K$-equivariant with respect to
the coadjoint action on $\mathfrak{k}^*$ and holomorphic with
respect to the complex structure $\ii$ on $M$. By (\ref{f23}) we
have $(\beta_3)_x + \ii (\beta_2)_x = 0$, and so applying the
derivative of $\mu_\cc$ at $x$ to this we obtain
$$0 = [\beta_3 + \sqrt{-1}\beta_2, \mu_\cc(x)] = [\beta_3 + \sqrt{-1}\beta_2,
\beta_2 + \sqrt{-1}\beta_3]
 = 2 [\beta_3, \beta_2].$$
\end{proof}

\begin{lemma} \label{l-5.2} If $ x\in M$ then
$\grad f_{23} (x) = 0$ if and only if
$$\mu_2(x)_x = \mu_3(x)_x = 0 . $$
\end{lemma}
\begin{proof}
Suppose $\grad f_{23}(x) = 0 $. 
 By Lemma \ref{l:wilkin} we have $[\beta_2, \beta_3] = 0 $,
so
$$ \mu_2(x)_x \cdot \ii \mu_3(x)_x = d\mu_1 (\mu_2(x)_x) \cdot \mu_3(x) $$
$$ = [\mu_2(x), \mu_1(x) ] \cdot \mu_3(x) $$
$$ = \mu_1(x) \cdot [\beta_2, \beta_3] = 0. $$
Then by (\ref{f23}) the result follows.
\end{proof}

Suppose that $\grad f_{23}(x) = 0$. Then by \eqref{f23}, the
Hessian $H_{23}$ of $f_{23}$ at $x$ is represented via the metric
as the self-adjoint endomorphism of $T_x M$ given by
\begin{equation} \label{e:5.1}
\frac{1}{2} H_{23} (\xi) = (\jj \beta_2 + \kk \beta_3) (\xi) + \jj
d\mu_2 (x) (\xi)_x + \kk d\mu_3(x) (\xi)_x
\end{equation}
for $\xi \in T_x M$. Here $\beta_2 = \mu_2(x)$ and $\beta_3 =
\mu_3(x)$ and $(\jj \beta_2 + \kk \beta_3)(\xi)$ denotes the
action of $\jj \beta_2 + \kk \beta_3$ on $T_x M$ induced from the
action of $K$ and the fact that $(\jj \beta_2 + \kk \beta_3)_x  =
0 $.

As before let $T$ be a maximal torus of $K$ with Lie algebra $\liet$ and for $j=2,3$ let
$\mathcal{B}_j$ be the set of $\beta \in \liets$ such that there is some $x \in X$
with $\mu_j(x)=\beta$ and $\beta_x=0$.

\begin{remit} \emph{When $M$ is compact then any $\beta \in \mathcal{B}_j$ is the closest
point to zero of the convex hull of a nonempty subset of the finite subset of $\liets$
which is the image under the $T$-moment map $\mu_j^T$ of the $T$-fixed point set in $M$
(see \cite{Kir} Lemma 3.13), and in particular $\mathcal{B}_j$ is finite.
In general it follows from the fact that $|\!|\mu_j|\!|^2$ is a minimally degenerate Morse
function in the sense of \cite{Kir} that $\mathcal{B}_j$ is at most countable.}
\end{remit}
 Let $\mathcal{B}_{23}$ be a set of
representatives of the Weyl group orbits in
$\mathcal{B}_2 \times \mathcal{B}_3$.
For each $(\beta_2, \beta_3) \in {\mathcal B}_{23}$
let
$$Stab(\beta_2, \beta_3) = Stab(\beta_2) \cap Stab (\beta_3) $$
where $Stab(\beta)$ denotes the stabilizer of $\beta$ under the
coadjoint action of $K$ on $\mathfrak{k}^* \cong \mathfrak{k}$.
Let
$T_{\beta_2 \beta_3} $ be the subtorus of $T$ generated by
$\beta_2$ and $\beta_3$, and let
$$ Z_{\beta_2,\beta_3} = \{ x \in M \mid
T_{\beta_2,\beta_3} \mbox{ fixes $x$ and } \mu_l(x) \cdot \beta_l =
\|\beta_l\|^2. ~~ l = 2, 3 \} $$

The argument of \cite{Kir} Lemma 3.15 shows the
following.
\begin{lemma}
The set of critical points for $f_{23}$ is the disjoint union of the
closed subsets $\{ C_{\beta_2, \beta_3}  \mid (\beta_2, \beta_3) \in
{\mathcal B}_{23} \} $ of $M$, where
$$C_{\beta_2, \beta_3} =
K (Z_{\beta_2, \beta_3} \cap {\mu_2}^{-1}(\beta_2) \cap
{\mu_3}^{-1}(\beta_3)).  $$
\end{lemma}
\begin{proof}
Suppose $\grad f_{23}(x) = 0 $. 
 By Lemma \ref{l:wilkin} and Lemma \ref{l-5.2}
 we have $$[\mu_2(x), \mu_3(x)] = 0 $$ and
$ \mu_2(x)_x =0 = \mu_3(x)_x $. Hence there exists $k \in K$ such
that $\mu_2(kx)$ and $\mu_3(kx)$ both lie in $\liet$ and
$$ \mu_2(kx)_{kx} =0 = \mu_3(kx)_{kx}. $$
Thus we have $ \mu_2(kx) =\beta_2$ and $ \mu_3(kx)= \beta_3 $ for some
$(\beta_2,\beta_3) \in \mathcal{B}_2 \times \mathcal{B}_3$, and multiplying
$k$ by a suitable element of the normalizer of $T$ in $K$ we may assume that
$(\beta_2,\beta_3) \in \mathcal{B}_{23}$. The result follows.
\end{proof}

To show that $f_{23}$ is a minimally degenerate Morse function in
the sense of  \S11 of \cite{Kir}, it is enough to find for each
$(\beta_2, \beta_3) \in {\mathcal B}_{23}$ a locally closed
submanifold $R_{\beta_2, \beta_3}$ of $M$ which contains
$C_{\beta_2, \beta_3}$ and has the following properties.
$R_{\beta_2, \beta_3}$ must be closed in a neighbourhood of
$C_{\beta_2, \beta_3}$. Moreover there should be a smooth
subbundle $U$ of $TM|_{R_{\beta_2, \beta_3}}$ such that
$TM|_{R_{\beta_2, \beta_3}}= U + TR_{\beta_2, \beta_3}$ and for
each $x \in C_{\beta_2, \beta_3}$ the restriction to $U_x$ of the
Hessian $H_{23}$ of $f_{23}$ at $x$ is nondegenerate. Then the
normal bundle to $ R_{\beta_2, \beta_3}$ in $M$ is isomorphic to a
quotient of $U$ and $U$ splits near $C_{\beta_2, \beta_3}$ as $U^+
+ U^- $ where the restriction of the Hessian $H_{23}$ to $U_x^+$
is positive definite and to $U_x^-$ is negative definite for all
$x$. One finds that the intersection of a small neighborhood of
$C_{\beta_2, \beta_3}$ in $M$ with the image of $U^+$ under the
exponential map $Exp: TM \to M$ satisfies the conditions for a
minimizing manifold for $f_{23}$ along $C_{\beta_2, \beta_3}$. By
\cite{Kir} (4.21) for $l = 2, 3$ we have
$$ K (Z_{\beta_l} \cap \mu_l^{-1} (\beta_l) ) \cong
 K \times_{Stab \beta_l} (Z_{\beta_l} \cap \mu_l^{-1} (\beta_l)  )
 $$
and $K Z_{\beta_l} \cong K \times_{Stab \beta_l} Z_{\beta_l} $
near $K(Z_{\beta_l} \cap \mu_l^{-1}(\beta_l)). $ Therefore
\begin{equation} \label{e:5.4}
C_{\beta_2,\beta_3} = K \times_{Stab(\beta_2,\beta_3)}
(Z_{\beta_2,\beta_3} \cap  \mu_3^{-1}(\beta_3))
\end{equation}
and
$K Z_{\beta_2, \beta_3} = K \times_{Stab(\beta_2, \beta_3)}
Z_{\beta_2, \beta_3} $
near $C_{\beta_2, \beta_3}$. In particular
 $K Z_{\beta_2, \beta_3}$ is smooth near
$C_{\beta_2,\beta_3}$
(cf. \cite{Kir} Cor 4.11).
If $x \in Z_{\beta_2, \beta_3} $
then $\mu_l(x) \cdot \beta_l = \|\beta_l\|^2 $
for $l = 2, 3 $ so
$f_{23} (x) \ge \|\beta_2\|^2 + \|\beta_3\|^2 $
with equality if and only if $\mu_2(x) = \beta_2$ and
$\mu_3(x) = \beta_3.$
Hence the restriction of $f_{23}$ to $KZ_{\beta_2,\beta_3}$
takes its minimum value precisely
on $C_{\beta_2,\beta_3}$. Thus to prove that
$f_{23}$ is minimally degenerate it suffices to prove

\begin{lemma} \label{l5.5}
For any $x \in Z_{\beta_2,\beta_3} \cap \mu_2^{-1}(\beta_2)
\cap \mu_3^{-1}(\beta_3) $ there is a complement $U_x$
to $T_x K Z_{\beta_2,\beta_3}$ in $T_x M$  varying smoothly
with $x$ such that $H_{23}$ restricts to a nondegenerate
bilinear form on $U_x$ and $U_{sx} = s (U_x) $ for all
$s \in Stab (\beta_2, \beta_3). $
\end{lemma}

\begin{proof}
~From (\ref{e:5.1}) it is easy to check that the subspace
$$T_x Z_{\beta_2,\beta_3} + \mathfrak{k}_x + \ii \mathfrak{k}_x + \jj
\mathfrak{k}_x + \kk \mathfrak{k}_x $$ of $T_x M$ is invariant under
the action of $H_{23}$ regarded as a self-adjoint endomorphism of
$T_x M$. Hence so is its orthogonal complement $V_x $ say, in $T_x
M$. If $\xi \in V_x$ then
$$ d \mu_2(x)(\xi) \cdot b = \xi \cdot \jj b_x = 0 $$
for all $b \in \mathfrak{k}$, so $d \mu_2(x)(\xi) = 0$
and similarly $d \mu_3(x)(\xi) = 0 $. Hence
$$ \frac{1}{2} H_{23} (\xi) = \jj \beta_2(\xi) + \kk \beta_3(\xi)$$
so that the restriction of $\frac{1}{2} H_{23}  $ to $V_x$ coincides
with the restriction of the Hessian of the function
$\mu_2^{(\beta_2)} + \mu_3^{(\beta_3)}$ on $M$, where
$\mu_l^{(\beta_l)}(x) $ is defined as $\mu_l (x) \cdot \beta_l.$ But
$Z_{\beta_2, \beta_3}$ is a union of components of the set of
critical points for this function. Because this function is a
nondegenerate Morse function on $M$ by the following lemma, the
restriction of $H_{23}$ to $V_x$ is nondegenerate.
\end{proof}

\begin{lemma} \label{l3.9} \cite[Lemma 3.9]{k6}
The function $\mu^{(\beta)}:=\mu_2^{(\beta_2)} + \mu_3^{(\beta_3)}$
is a nondegenerate Morse-Bott function.
\end{lemma}

\begin{proof} 
The gradient of $\mu^{(\beta)}$ at $x\in X$ is
\[
\grad  \mu^{(\beta)}(x)=\jj (\beta_2)_x + \kk (\beta_3)_x.
\]
Since $$\jj(\beta_2)_x\cdot \kk(\beta_3)_x=\ii (\beta_2)_x\cdot
(\beta_3)_x=d\mu_1((\beta_3)_x)\cdot \beta_2=[\beta_3,\mu_1(x)]\cdot
\beta_2=\mu_1(x)\cdot [\beta_2,\beta_3]=0,$$
 the set of critical
points is precisely the fixed point set of the subtorus $T_\beta$
which is the closure of the subgroup generated by $\exp \rr\beta_2$
and $\exp \rr\beta_3$. Thus every connected component is a
submanifold of $M$. So it suffices to prove that its Hessian at any
critical point is nondegenerate in directions orthogonal to the
critical set. Let $H_2, H_3$ and $H = H_2 + H_3$ be the Hessians at $x$ for
$\mu_2^{(\beta_2)}$, $\mu_3^{(\beta_3)}$ and $\mu^{(\beta)}$. Choose
local coordinates near $x$ such that $T_\beta$ acts linearly and the
metric is given by a matrix $\rho_{pq}(y)$ with
$\rho_{pq}(x)=\delta_{pq}$. Since
\[
0=\jj (\beta_2)_y\cdot \kk(\beta_3)_y=\grad
\mu_2^{(\beta_2)}(y)\cdot \grad \mu_3^{(\beta_3)}(y)
\]
for all $y\in M$, we have
\[
\sum_{p,q} \frac{\partial \mu_2^{(\beta_2)}}{\partial
y_q}(y)\rho^{-1}_{pq}(y)\frac{\partial \mu_3^{(\beta_3)}}{\partial
y_p}(y) =0
\]
for all $y$, and since
\[
\grad \mu_2^{(\beta_2)}(x)=0=\grad \mu_3^{(\beta_3)}(x),
\]
differentiating twice gives
\[ H_2H_3+H_3H_2=0\]
which gives us $H^2=(H_2+H_3)^2=H_2^2+H_3^2$. If $H\xi=0$,
\[
0=\|H\xi\|^2=\langle H^2\xi,\xi\rangle = \langle
H_2^2\xi,\xi\rangle+\langle
H_3^2\xi,\xi\rangle=\|H_2\xi\|^2+\|H_3\xi\|^2
\]
and thus $H_2\xi=0=H_3\xi$. Therefore, if $H\xi=0$, $\xi$ is tangent
to the fixed point set of $T_\beta$ because $\mu_2^{(\beta_2)}$ and
$\mu_3^{(\beta_3)}$ are nondegenerate.
%
%
\end{proof}

\newcommand{\stab}{{\rm Stab}}
{\em Conclusion of Proof of Lemma \ref{l5.5}:}
It follows from Lemma \ref{l3.9} that the restriction of $H_{23}$ to $V_x$ is
nondegenerate.
Therefore it suffices to show that there is a complement
$W_x$ to $T_x K Z_{\beta_2, \beta_3}$ in
$$T_x Z_{\beta_2,\beta_3} + \mathfrak{k}_x +\ii \mathfrak{k}_x + \jj
\mathfrak{k}_x + \kk \mathfrak{k}_x $$ varying smoothly with $x$
such that $H_{23}$ is nondegenerate on $W_x$ and $W_{sx} = s(W_x)$
for all $s \in Stab (\beta_2, \beta_3)$. For then we may take $U_x =
V_x + W_x$.


From (\ref{e:5.1}), we obtain
\[
\frac12 H_{23}(\jj a_x)=-[\beta_2,a]_x+\jj A_x(a)_x
-\ii [\beta_3,a]_x + \kk [\beta_1,a]_x
\]
\[
\frac12 H_{23}(\kk a_x)=-[\beta_3,a]_x+\kk A_x(a)_x
+\ii [\beta_2,a]_x - \jj [\beta_1,a]_x
\]
where $A_x:  \mathfrak{k} \to \mathfrak{k} $ is the self-adjoint
endomorphism defined by
$$ A_x(a) \cdot b = a_x \cdot b_x
$$
for all $a, b \in \mathfrak{k}$.

Hence, via the projection
\begin{equation}\label{eq-yh1} \mathfrak{k} \oplus
\mathfrak{k} \oplus \mathfrak{k} \oplus \mathfrak{k}
\longrightarrow\mathfrak{k}_x  + \ii \mathfrak{k}_x +\jj\mathfrak{k}_x +
\kk\mathfrak{k}_x,
\end{equation}
the endomorphism  $\frac{1}{2} H_{23} $ lifts to the endomorphism of $
\mathfrak{k} \oplus \mathfrak{k} \oplus \mathfrak{k} \oplus
\mathfrak{k} $ given by the
matrix 
\begin{equation} \label{e:sign} {\cM} = \left \lbrack \begin{array}{cccc}
0 & 0 &  - Ad \beta_2 & - Ad \beta_3 \\
0 & 0 & - Ad \beta_3 & Ad \beta_2 \\
0 &  2 Ad \beta_3 &  A_x &  - Ad \beta_1 \\
0 & - 2 Ad \beta_2 & Ad \beta_1  &  A_x \end{array} \right \rbrack
\end{equation}
acting as left multiplication on column vectors.

Since $(\beta_2)_x = (\beta_3)_x = 0 $ and $\mu_1$ is
$K$-equivariant, we have $[\beta_2, \beta_1] = [\beta_3, \beta_1] =
0 $, and also $[\beta_2, \beta_3] = 0 $.
Hence $Ad \beta_1$, $Ad
\beta_2$ and $Ad \beta_3$ all commute, and so do $A_x,$ $Ad \beta_2$
and $Ad \beta_3$.

To identify
the kernel of $\mathcal{M}$,
consider the  set of $(\xi,\zeta, \eta)$ satisfying $$\mathcal{M}'
\left \lbrack
\begin{array}{c} \xi \\ \zeta \\ \eta \end{array} \right \rbrack = 0 $$
where 
the matrix $\mathcal{M}' $ consists of the second, third and fourth rows
and columns of the matrix $\mathcal{M}$, in other words
\begin{equation}\mathcal{M}' =
\left \lbrack \begin{array}{ccc}
 0 & - Ad \beta_3 &  Ad \beta_2 \\
 2 Ad \beta_3 & A_x & -Ad \beta_1 \\
 -2 Ad \beta_2 & Ad \beta_1 & A_x  \end{array} \right \rbrack.
\end{equation}
For simplicity, write $A = A_x$,  $B_1 = Ad \beta_1$,$B_2 = Ad \beta_2$,
$B_3 = Ad \beta_3$.

The adjugate matrix\footnote{Note that for matrices of complex numbers the product of
a matrix $M$ and its adjugate is the identity matrix times
the determinant of $M$. This is consistent with (\ref{e:madj}),
except that we do not know  that $B_1$ and $A$ commute -- if they do
commute, the matrix in
(\ref{e:madj}) is simply the determinant of $\mathcal{M}'$ times
the identity matrix.} of $\mathcal{M}'$ is
\newcommand{\madj}{{\mathcal{M}'}_{adj}}
$$ \madj = \left \lbrack \begin{array}{ccc}
A^2 + (B_1)^2 & B_1 B_2 + A B_3 & B_1 B_3 - A B_2 \\
2(B_1 B_2 - A B_3) & 2(B_2)^2 &  2B_2 B_3 \\
2(B_1 B_3 + A B_2) & 2B_2 B_3 & 2(B_3)^2 \end{array} \right
\rbrack
$$
By left multiplying ${\mathcal{M}}'$ by its
adjugate matrix, we obtain

\begin{equation} \label{e:madj}
\left \lbrack \begin{array}{ccc}
2A ((B_2)^2 + (B_3)^2) & (B_1 A-A B_1) B_2 & (B_1 A - A B_1) B_3 \\
0 & 2A ((B_2)^2 + (B_3)^2 ) & 0 \\
0 & 0& 2A ((B_2)^2 + (B_3)^2 ) \end{array} \right \rbrack \left
\lbrack \begin{array}{c} \xi \\ \zeta \\ \eta  \end{array} \right
\rbrack = 0.
\end{equation}

Since $A$, $B_2$ and $B_3$ commute and $A$ is self-adjoint while
$B_2$ and $B_3$ are skew-adjoint,
these matrices
can be simultaneously diagonalized with real and purely imaginary
eigenvalues respectively. The Lie algebra $stab (x)$ of the
stabilizer of $x$ in $K$ is contained in
$stab(\beta_2,\beta_3)=stab (\beta_2) \cap stab (\beta_3)$ where
$stab (\beta_i)$ denotes the Lie algebra of $\stab \beta_i$. So
the kernel of $A$ is contained in the intersection of the kernel
of $B_2$ and that of $B_3$.

The equations given by the lower two rows of (\ref{e:madj})
tell us that
$\zeta$ and $\eta$ are in the
intersection of the  kernels
 of  $B_2 $ and $B_3$.
 Now consider the equation determined by  the top row:

$$ 2A (B_2^2 + B_3^2) \xi +(B_1A - AB_1) B_2 \zeta + (B_1A - AB_1)B_3 \eta
  = 0.$$
%
Since $\zeta$ and $\eta$ are in the kernel of $B_2$ and $B_3$,
 this reduces to

$$ A (B_2^2 + B_3^2) \xi = 0$$
so we conclude that  $\xi$ is  also  in the kernel of $B_2$ and
$B_3$. Therefore, the kernel of $\mathcal{M}$ is contained in the
subspace
$$L:=\mathfrak{k}
\oplus  stab (\beta_2, \beta_3)   \oplus stab (\beta_2, \beta_3)
\oplus stab (\beta_2, \beta_3). $$ In other words, if we let
$$\phi:\mathfrak{k}^{\oplus 4}\longrightarrow
\mathfrak{k}^{\oplus 4}$$ denote the linear transformation defined
by $\mathcal{M}$, we proved that $\phi^{-1}(0)\subset L$. By a
similar calculation, we can furthermore prove that
$$\phi^{-1}(L)\subset L.$$ Indeed, if $(\alpha,\xi,\zeta,\eta)\in
\phi^{-1}(L)$, i.e. $$\mathcal{M} \left \lbrack
\begin{array}{c} \alpha \\ \xi \\ \zeta \\ \eta \end{array} \right \rbrack =
\left \lbrack
\begin{array}{c} a \\ b \\ c \\ d \end{array} \right \rbrack $$
with $b,c,d\in stab(\beta_2,\beta_3)$, we multiply the $4\times 4$
matrix $\widetilde{\mathcal{M}}'_{adj}$ which contains
$\mathcal{M}'_{adj}$ as the lower right $3\times 3$ submatrix and
whose other entries are zero. Because $b,c,d\in
stab(\beta_2,\beta_3)$, the right hand side
$\widetilde{\mathcal{M}}'_{adj} (a,b,c,d)^t$ equals
$(0,(A^2+B_1^2)b,0,0)^t$. The left hand side is the same as the
$4\times 4$ matrix which contains the left hand side of
\eqref{e:madj} as the lower right submatrix and whose other
entries are zero. By comparing the third and fourth entries, we
obtain $\zeta,\eta\in stab(\beta_2,\beta_3)$. But then the second
entry $b$ of $\mathcal{M}(\alpha,\xi,\zeta,\eta)^t$ is zero by the
definition of $\mathcal{M}$. This implies that $(A^2+B_1^2)b=0$
and so we obtain the equation \eqref{e:madj}, which tells us that
$\xi\in stab(\beta_2,\beta_3)$ as well. Hence,
$(\alpha,\xi,\zeta,\eta)\in L$ as desired.

It is obvious from the commutativity of $A$ (resp. $B_1$) with
$B_2, B_3$ that $\phi(L)\subset L$. So we obtain the induced
homomorphism
\[
\bar \phi:\mathfrak{k}^{\oplus 4}/L\longrightarrow
\mathfrak{k}^{\oplus 4}/L
\]
which is injective because $\phi^{-1}(L)\subset L.$ Hence
$\bar\phi$ is an isomorphism.

The projection \eqref{eq-yh1} maps $L$ into $T_x K Z_{\beta_2,
\beta_3}$ because $Z_{\beta_2, \beta_3} $ is a $Stab(\beta_2,
\beta_3)$-invariant hyperk\"ahler submanifold. If we further
compose \eqref{eq-yh1} with the projection to the orthogonal
complement $W_x$ of $T_x K Z_{\beta_2,\beta_3}$, we obtain a
commutative diagram
\[
\xymatrix{ \mathfrak{k}^{\oplus 4}\ar[r]\ar[d]
&\mathfrak{k}^{\oplus
4}\ar[d]\\
W_x\ar[r] &W_x }
\]
which induces a commutative diagram
\[
\xymatrix{ \mathfrak{k}^{\oplus 4}/L\ar[r]^{\cong}\ar[d]
&\mathfrak{k}^{\oplus
4}/L\ar[d]\\
W_x\ar[r] &W_x .}
\]
It is now obvious that the bottom arrow is an isomorphism.
Therefore the restriction of $H_{23}$ to $W_x$ is nondegenerate.
So the proof of Lemma \ref{l5.5} is complete by taking $W_x$ to be
the orthogonal complement of $T_x K Z_{\beta_2,\beta_3}$.


This completes the proof that $f_{23}$ is a minimally degenerate
Morse function and so we deduce that $f_{23}$ is an equivariantly
perfect Morse function if it is flow-closed.
%

\section{Hyperk\"ahler quotients of quaternionic vector spaces}
In this section 
we consider linear actions on hyperk\"ahler vector spaces.

Let $V=\cc^{2n}=T^*\cc^n$ be a \hk vector space. Suppose we have a
Hamiltonian \hk linear action of a compact Lie group $K$ on $V$. Fix
a \emph{general \hk frame} $\{\ii,\jj,\kk\}$. Let $\mu$ (resp.
$\mu_\cc$) be the hyperk\"ahler (resp. complex) moment map for this
action. From the surjectivity criterion in Section
\ref{s:s4}, 
we
obtain the following.

\begin{proposition} \label{prop:linear2} Suppose $-\|\mu_\cc\|^2$ on $V$ is
flow-closed.
If $0$ is a regular value of $\mu$, the restriction map
$$H^*_K(V)=H^*_K\longrightarrow H^*_K(\mu^{-1}(0))\cong
H^*(\mu^{-1}(0)/K)$$ is surjective.
\end{proposition}

This proposition is relevant to Hilbert schemes of points in $\cc^2$,
ADHM spaces, hypertoric manifolds and quiver varieties, etc.
\begin{proof}
Sjamaar has
proved that $-\|\mu_1\|^2$  is flow-closed on $V$ and hence
on $W=\mu_\cc^{-1}(0)$
(\cite[(4.6)]{Sja-Conv}), so the
hypotheses of Theorem \ref{ThSurjCrit}
are satisfied.
\end{proof}

We believe that it is likely that the assumption on flow-closedness of $-\|\mu_\cc\|^2$
in Proposition \ref{prop:linear2} is always satisfied.
\begin{conjecture}\label{flowclosed-conj}
For a linear \hk action on a \hk vector space of a compact Lie
group $K$, $-\|\mu_\cc\|^2$ is flow-closed.
\end{conjecture}

Although it is purely a calculus problem, we do not know how to
prove this conjecture, though we have managed to prove the circle case.
\begin{remit}  \label{rk5.3} \emph{ Certain other cases  follow from the circle
case. For example if the conjecture is true for the action of
$K_1$ on $V_1$ and for the action of $K_2$ on $V_2$ then it is true
for the action of $K_1 \times K_2$ on $V_1 \oplus V_2$, so the conjecture
follows for suitable torus actions.
}
\end{remit}
\begin{proposition} \label{p:flowclosedcircle}
Conjecture \ref{flowclosed-conj} is true for the circle group
$U(1)$.
\end{proposition}

Let $x,y\in\cc^n$. We denote by $\bar x,\bar y$ the complex
conjugates of $x,y$. By a suitable change of variables and
ignoring the variables for which the weights are zero, we may
assume
$$\mu_\cc(x,y)=x\cdot y -c=\sum_{i=1}^n x_iy_i -c$$ for some $c\in
\cc$. Let
$$f(x,y)=|\mu_\cc(x,y)|^2\ge 0$$ and let $\gamma(t)$ be the
integral curve of the vector field $-\mathrm{grad} f$ from a given
$(x_0,y_0)\notin \mu_\cc^{-1}(0)$.

\begin{lemma} $\{\gamma(t)\}$ is contained in a compact
subset of $\cc^{2n}$.
\end{lemma}
\begin{proof} Let
$$\rho(x,y)=|x|^2+|y|^2=\sum_{i=1}^n(|x_i|^2+|y_i|^2)\ge 0.$$ By
direct computation, we know
\[
\mathrm{grad} f=2\mu_\cc \left(\begin{matrix} \bar y\\ \bar x
\end{matrix}
\right)\quad \text{and}\quad \mathrm{grad} \rho =2\left(\begin{matrix} x\\
y
\end{matrix} \right) .
\]

If $\gamma(t) \in \mu_\cc^{-1}(0)$ for some $t$ then $\gamma$ is
constant and there
is nothing to prove. So we assume $\gamma(t)\notin
\mu_\cc^{-1}(0)$ for all $t\ge 0$. On $\cc^{2n}-\mu_\cc^{-1}(0)$,
$\sqrt{f}=|\mu_\cc|\ge 0$ is differentiable and $$\mathrm{grad}
\sqrt{f}=\frac1{2\sqrt{f}}\mathrm{grad} f.$$

If we denote by $\langle ,\rangle$ the ordinary inner product of
$\mathbb{R}^{4n}\cong\cc^{2n}$, we have
\[
\langle \mathrm{grad}\sqrt{f}, \mathrm{grad} f\rangle =
\frac1{2\sqrt{f}}|\mathrm{grad} f|^2=
\frac{4|\mu_\cc|^2}{2\sqrt{f}} (|x|^2+|y|^2)=2\rho \sqrt{f},
\]
\[
\langle \mathrm{grad}\rho, \mathrm{grad} f\rangle = 8\,
\mathrm{Re}\, \mu_\cc (\bar x\cdot \bar y)=8\, \mathrm{Re}\,
\mu_\cc (\bar \mu_\cc+\bar c)\]
\[
=8\, \mathrm{Re}\,(f+\bar c \mu_\cc)\ge 8f-8|c|\sqrt{f}.
\]
Therefore,
\[
\langle \mathrm{grad} (\rho+\sqrt{f}), \mathrm{grad} f\rangle \ge
8f+2\sqrt{f} (\rho-4|c|)
\]
and hence $(\rho+\sqrt{f})(\gamma(t))$ is decreasing whenever
$\rho=|x|^2+|y|^2\ge 4|c|$.

Now we claim, for all $t\ge 0$,
\[
(\rho+\sqrt{f})(\gamma(t))\le \mathrm{max}\{ 4|c|+\sqrt{f_0},
\rho_0+\sqrt{f_0}\}
\]
where $\rho_0=\rho(x_0,y_0)$ and $f_0=f(x_0,y_0)$. The proposition
follows from this claim because $\rho=|x|^2+|y|^2$ is bounded
along the curve $\gamma(t)$ since $\rho+\sqrt{f}$ is bounded and
$\sqrt{f}\ge 0$.

It remains to show the claim above. Since $f$ is decreasing along
$\gamma(t)$ by the definition of $\gamma(t)$, $\sqrt{f}$ is also
decreasing for all $t\ge 0$. If $\rho(\gamma(t))\le 4|c|$ for some
$t\ge 0$, $\rho+\sqrt{f}\le 4|c|+\sqrt{f_0}$ at $\gamma(t)$.

If $\rho(\gamma(t))>4|c|$ for some $t \ge 0$, then there are two
possibilities. \begin{enumerate} \item [(1)] $\gamma$ on $[0,t]$
stays outside of the sphere $\rho = 4|c|$: in this case, since
$\rho+\sqrt{f}$ is decreasing on $[0,t]$, $\rho+\sqrt{f}\le
\rho_0+\sqrt{f_0}$.

\item [(2)] $\gamma$ on $[0,t]$ moves into and out of the sphere
$\rho = 4|c|$: then let $\tau$ be the greatest element of $[0,t]$ satisfying
$\rho(\gamma(\tau))\le 4|c|$. 
Then $\tau < t$
and because
$\rho(\gamma)>4|c|$ on the interval $(\tau,t]$, it follows that
$\rho+\sqrt{f}$ at
$\gamma(t)$ is less than $\rho+\sqrt{f}$ at $\gamma(\tau)$, which
is less than or equal to $4|c|+\sqrt{f_0}$ since
$\rho(\gamma(\tau))\le 4|c|$ and $\sqrt{f}$ is decreasing. Hence
$\rho+\sqrt{f}\le 4|c|+\sqrt{f_0}.$
\end{enumerate}
%
So we are done.
\end{proof}

\begin{example} \emph{ Consider the action of $\cc^{2n}=\cc^n\oplus \cc^n$ with
weights $1$ and $-1$ as usual. For $x,y\in \cc^n$,
$\mu_\cc(x,y)=\sum x_iy_i-c$ for $c\ne 0$. There is only one
non-minimal critical point which is the origin. By direct
computation, the Hessian at the origin has $2n$ positive and $2n$
negative eigenvalues. Hence we deduce that the Poincar\'e
polynomial of the \hk quotient is
\[ \frac1{1-t^2} -
\frac{t^{2n}}{1-t^2}=P_t(\pp^{n-1})=P_t(T^*\pp^{n-1}).
\]
}
\end{example}
\medskip


\bibliographystyle{amsplain}

\begin{thebibliography}{10}

\bibitem{ADHM} M. Atiyah, V. Drinfeld, N. Hitchin and Y. Manin.
Construction of instantons. {\em Phys. Lett. A} {\bf 65} (1978), 185--187.

\bibitem{BBDG} A. Beilinson, J. Bernstein and P. Deligne. Faisceaux pervers. {\em  Analysis and topology on singular spaces, I (Luminy, 1981)},  5--171, Astérisque, 100, Soc. Math. France, Paris, 1982.

\bibitem{BL} J. Bernstein and V. Lunts. Equivariant sheaves and functors. {\em Lecture
Notes in Math.} {\bf 1578}, Springer-Verlag 1994.


\bibitem{Biqu} O. Biquard and P. Gauduchon. La metrique
hyperkahlerienne des orbites coadjointes de type symetrique d'un
groupe de Lie complexe semi-simple. {\em C. R. Acad. Sci. Paris Ser.
I Math.} {\bf 323} (1996), no.12, 1259--1264.

\bibitem{GuillStern} V. Guillemin and S. Sternberg, {\em Symplectic
techniques in physics}. Cambridge University Press, 1984.


\bibitem{Hit1} N. Hitchin, A. Karlhede,
U. Lindstr\"om and M. Rocek, Hyperk\"ahler quotients and
supersymmetry. {\em Commun. Math. Phys.} {\bf 108} (1987) 535--589.

\bibitem{KLS} D. Kaledin, M. Lehn and C. Sorger. Singular symplectic moduli spaces.
{\em Invent. Math.} {\bf 164} (2006), no. 3, 591--614.

\bibitem{KY}
Y.-H. Kiem and S.-B. Yoo.
The stringy E-function of the moduli space of Higgs bundles with trivial
determinant. {\em Math. Nachr.} {\bf 281} (2008),  817--838.

\bibitem{KY2}
Y.-H. Kiem and S.-B. Yoo. The intersection cohomology of the
moduli space of Higgs bundles with trivial determinant. In
preparation.

\bibitem{Kir} F. Kirwan. {\em Cohomology of quotients in symplectic and
algebraic geometry}. Princeton University Press, 1984.

\bibitem{Kir2} F. Kirwan. Partial desingularisation of quotients of
nonsingular varieties and their Betti numbers. {\em Ann. Math.}
{\bf 122} (1985) 41--85l

\bibitem{k6} F. Kirwan. On the topology of hyperk\"ahler quotients.
 IHES preprint IHES/M/85/35, 1985.
\bibitem{Mostow} G.D. Mostow. On a conjecture of Montgomery, {\em
Ann. Math.} {\bf 65} (1957) 513--516.

\bibitem{konno} H. Konno, Cohomology rings of toric hyperk\"ahler
manifolds. {\em International J. of Math.} {\bf 11} (2000) 1001--1026.

\bibitem{konno2} H. Konno, Variation of toric hyperk\"ahler
manifolds. {\em International J. of Math.} {\bf 14} (2003) 289--311.

\bibitem{Kron} P. Kronheimer. A \hk structure on coadjoint orbits of
a semisimple complex group. {\em J. Lond. Math. Soc.} (2) {\bf 42}
(1990), no.2, 193--209.

\bibitem{quiver} H. Nakajima. Instantons on ALE spaces, quiver varieties and
Kac-Moody algebras, {\em Duke Math. J.} {\bf 76} (1994)
365--416.

\bibitem{Sjamaar} R. Sjamaar. Holomorphic slices,
symplectic reduction and multiplicities of representations.
{\em Ann. Math.} {\bf 141} (1995) 87--129.

\bibitem{Sja-Conv} R. Sjamaar. Convexity properties of the
moment map re-examined. {\em Adv. Math.} {\bf 138}
(1998) 46--91.

\bibitem{Woolf} J. Woolf. The decomposition theorem and the intersection cohomology of quotients in algebraic geometry.  {\em J. Pure Appl. Algebra}  {\bf 182}  (2003),  317--328.
\end{thebibliography}

\end{document}